\documentclass{amsart}

\usepackage{amsmath, amssymb,epic,graphicx,mathrsfs,enumerate}
\usepackage[all]{xy}

\usepackage{latexsym}
\usepackage{longtable}
\usepackage{epsfig}
\usepackage{hhline}
\usepackage{ytableau}

\newtheorem{proposition}{Proposition}[section]

\newtheorem{lemma}[proposition]{Lemma}

\newtheorem{theorem}[proposition]{Theorem}

\DeclareMathOperator{\Sym}{Sym}
\DeclareMathOperator{\Alt}{Alt}
\DeclareMathOperator{\Irr}{Irr}

\begin{document}

\title{Alternating groups as products of four conjugacy classes}
\date{}

\author[Martino Garonzi]{Martino Garonzi}
\address[Martino Garonzi]{Departamento de Matem\'atica, Universidade de Bras\'ilia, Campus Universit\'ario Darcy Ribeiro, Bras\'ilia-DF, 70910-900, Brazil}
\email{mgaronzi@gmail.com}

\author{Attila Mar\'oti}
\address[Attila Mar\'oti]{Alfr\'ed R\'enyi Institute of Mathematics, Re\'altanoda utca 13-15, H-1053, Budapest, Hungary}
\email{maroti.attila@renyi.hu}

\thanks{The first author acknowledges the support of Funda\c{c}\~{a}o de Apoio \`a Pesquisa do Distrito Federal (FAPDF) - demanda espont\^{a}nea 03/2016, and of Conselho Nacional de Desenvolvimento Cient\'ifico e Tecnol\'ogico (CNPq) - Grant numbers 302134/2018-2, 422202/2018-5. The work of the second author on the project leading to this application has
received funding from the European Research Council (ERC) under the
European Union's Horizon 2020 research and innovation programme
(grant agreement No. 741420). He was also
supported by the National Research, Development and Innovation Office
(NKFIH) Grant No.~K115799 and Grant No.~K132951.}


\begin{abstract}
Let $G$ be the alternating group $\Alt(n)$ on $n$ letters. We prove that for any $\varepsilon > 0$ there exists $N = N(\varepsilon) \in \mathbb{N}$ such that whenever $n \geq N$ and $A$, $B$, $C$, $D$ are normal subsets of $G$ each of size at least $|G|^{1/2+\varepsilon}$, then $ABCD = G$.
\end{abstract}

\maketitle

\section{Introduction}

Given two subsets $A,B$ of a group $G$ we denote by $AB$ the set of products $ab$ where $a \in A$, $b \in B$. A subset $A$ of $G$ is called normal if $gAg^{-1} = A$ for all $g \in G$. Clearly, a subset of $G$ is normal if and only if it is a union of conjugacy classes. Observe that if $A$ and $B$ are normal sets, then $AB=BA$. 

The covering number of a nontrivial conjugacy class $C$ of a finite nonabelian simple group $G$ is the minimum positive integer $k$ such that $C^k=G$. Brenner \cite{Brenner} showed that almost all conjugacy classes of the alternating group $\Alt(n)$ have covering number at most $4$, and observed that there are classes with covering number $4$, for example the class of fixed-point-free involutions (see the penultimate paragraph of the Introduction).

Larsen and Shalev \cite[Theorem 1.13]{LarsenShalev} proved that an element $g$ of the symmetric group $\Sym(n)$ satisfies $(g^{\Sym(n)})^2=\Alt(n)$ with probability tending to $1$ as $n \to \infty$. (Here and throughout the paper $x^G$ denotes the conjugacy class of an element $x$ in a finite group $G$.) For a related result see \cite[Theorem 1.20]{LarsenShalev}. Larsen and Shalev also proved \cite[Theorem 1.14]{LarsenShalev} that if $n$ is sufficiently large, then any element $g \in \Sym(n)$ with at most $n/5$ fixed points satisfies $(g^{\Sym(n)})^4 = \Alt(n)$.

In this paper we take a different approach, considering the product of possibly distinct normal sets. Larsen, Shalev and Tiep \cite{LST} proved that if $\varepsilon > 0$ is a constant, then for sufficiently large $n$ the following holds: whenever $A,B$ are two normal subsets of $G=\Alt(n)$ of size larger than $\varepsilon |G|$, then $AB$ contains every nontrivial element of $G$, and they proved that the same holds for simple groups of Lie type of bounded rank. In this context, a subset is large if it has size at least the size of $G$ multiplied by a universal positive constant (less than $1$). Observe that using their result it is easy to see that, if $A,B,C$ are large normal subsets of $G$, then $ABC=G$. We are interested in studying largeness related to the size of $G$ raised to a constant $\gamma$.

Let $G=\Alt(n)$ be the alternating group on $n$ letters. In \cite[Theorem 1.3]{MP} it is proved that there exists $\gamma$ with $0 < \gamma < 1$ such that whenever $8$ normal subsets of $G$ have size at least $|G|^{\gamma}$, their product is $G$. It was asked if the same holds with less than $8$ normal sets. In this paper we prove that the result holds for $4$ normal sets and that if $\gamma$ is close to $1/2$, then this is best possible.

\begin{theorem} \label{main}
For any $\varepsilon > 0$ there exists $N \in \mathbb{N}$ such that whenever $n > N$ and $A,B,C,D$ are normal subsets of $G=\Alt(n)$ such that all of the numbers $|A||B|$, $|A||C|$, $|A||D|$, $|B||C|$, $|B||D|$, $|C||D|$ are at least $|G|^{1+\varepsilon}$, then $ABCD=G$.
\end{theorem}

In particular Theorem \ref{main} applies to the case in which the four normal subsets have size not less than $|G|^{1/2+\varepsilon}$, improving \cite[Theorem 1.3]{MP} in the case of alternating groups. The question of whether there exists $\gamma$ with $\gamma < 1$ such that, whenever $A,B,C$ are normal subsets of $G=\Alt(n)$ with $|A|,|B|,|C| \geq |G|^{\gamma}$, then $ABC=G$ is still open, however Theorem \ref{main} goes in this direction, since one of the four classes is allowed to be very small. If we interpret largeness in the sense of Larsen, Shalev and Tiep, then the product of any three large normal sets equals $G$, as seen above.

Theorem \ref{main} is best possible in the following sense. Let $n$ be a multiple of $4$, let $G = \Alt(n)$ and let $x$ be a fixed-point-free involution in $G$. Let $C$ be the conjugacy class of $x$ in $G$. Then, using the fact that $(n/3)^n \leq n! \leq (n/2)^n$ whenever $n \geq 6$, which can be easily deduced from Stirling's inequalities, one may see that if $n$ is sufficiently large, then $|C| = (n-1)!! = (n-1)(n-3) \cdots 3 \cdot 1 \geq (2/3)^n |G|^{1/2}$. This implies that $|C|$ is arbitrarily close to $|G|^{1/2}$ in the sense that for every $\varepsilon > 0$ there exists $N \in \mathbb{N}$ such that $|C| \geq |G|^{1/2-\varepsilon}$ for every $n \geq N$. However, as shown by Brenner in \cite[Lemma 3.06]{Brenner}, $C^3 \neq G$ and $C^4=G$. See also \cite{V}.

The paper is organized as follows. In Section \ref{delta} we introduce a useful tool by Dvir and Rodgers used to decide whether a product of two conjugacy classes of $\Sym(n)$ contains the $n$-cycles, for $n$ odd, and the $(n-1)$-cycles, for $n$ even, based on the number of disjoint cycles of an element in each class (Theorem \ref{dr}). We then relate this to our context (Lemma \ref{deltac}). In Section \ref{chsum} we recall known facts about character sums and how to apply them to products of conjugacy classes. In Section \ref{symalt} we recall how to compute character values for the symmetric and alternating groups. In Section \ref{proofmain} we finish the proof of Theorem \ref{main}.

\section{The $\delta$ of a conjugacy class} \label{delta}

Given a conjugacy class $C$ of $\Alt(n)$ or $\Sym(n)$, let $\delta(C) := n-t$, where $t$ is the number of disjoint cycles of an element of $C$. Dvir \cite{Dvir} proved a fact reformulated by Rodgers \cite[Theorem 1.1(iii)]{R}, which we will state here a particular case of. Denote the set of all $n$-cycles in $\Sym(n)$ by $O_n$, and denote the set of all $(n-1)$-cycles in $\Sym(n)$ by $O_{n-1}$. Observe that if $C$ is a conjugacy class of $\Sym(n)$ contained in $\Alt(n)$ then $\delta(C)$ is even.

\begin{theorem} \label{dr}
Let $A$, $B$ be two conjugacy classes of $\Sym(n)$ contained in $\Alt(n)$.
\begin{enumerate}
\item If $n$ is odd and $\delta(A)+\delta(B)>n-1$, then $O_n \subseteq AB$. 
\item If $n$ is even and $\delta(A)+\delta(B)>n$, then $O_{n-1} \subseteq AB$.
\end{enumerate}
\end{theorem}

Using Dvir's results, Rodgers proved the following \cite[Lemma 2.2]{R}:
\begin{lemma} \label{on2an}
$O_n^2 = O_{n-1}^2 = \Alt(n)$ for every $n$.
\end{lemma}

Let $G=\Alt(n)$. In order to apply Theorem \ref{dr} in our context, we need to translate the condition $|C| \geq |G|^{\gamma}$, for a conjugacy class $C$ of $G$ and a constant $\gamma$, into a lower bound for $\delta(C)$. This is precisely what we do in the following lemma.

\begin{lemma} \label{deltac}
For every $\gamma$ and $\varepsilon$ with $0 < \gamma < 1$ and $0 < \varepsilon < 1$ there exists $N \in \mathbb{N}$ such that for every $n \geq N$, whenever $x \in G = \Alt(n)$ satisfies $|x^G| \geq |G|^{\gamma}$, then $\delta(x^G) > (\gamma-\varepsilon) n.$
\end{lemma}

\begin{proof}
Let $\gamma$ and $\varepsilon$ be arbitrary positive real numbers less than $1$. Choose $\varepsilon_1$ with $0 < \varepsilon_1 < \varepsilon/(1-\gamma+\varepsilon)$, and observe that $\varepsilon_1 < 1$. Let $x \in G = \Alt(n)$ be such that $|x^G| \geq |G|^{\gamma}$. Let $r_i$ be the number of $i$-cycles in the cycle structure of $x$ for $i=1,\ldots,t$, where $t$ is a fixed positive integer such that $$\frac{1}{2} \varepsilon_1 (1-\gamma+\varepsilon) < \frac{1}{t+1} < \varepsilon_1 (1-\gamma+\varepsilon).$$ Set $s:=\sum_{i=1}^t r_i$. The number of disjoint cycles of $x$ is at most $n/(t+1)+s$, thus $$\delta(x^G) \geq n-\frac{n}{t+1}-s > n-\varepsilon_1(1-\gamma+\varepsilon)n-s.$$ In order to prove the lemma, it is sufficient to show that $s < n (1-\varepsilon_1) (1-\gamma+\varepsilon)$ for every sufficiently large $n$.

Observe that we may assume that $s$ is not bounded above by a fixed universal constant. In particular we assume that $s > 3t$.

By plugging $x_1=\ldots=x_t=1$ into the well-known identity of multinomial coefficients $$(x_1+\cdots+x_t)^s = \sum_{k_1+\ldots+k_t = s} \frac{s!}{k_1! \cdots k_t!} \prod_{i=1}^t x_i^{k_i},$$ we obtain $s!/t^s \leq \prod_{i=1}^t r_i!$.

An easy application of Stirling's inequality gives $s! \geq (s/3)^s$. Since $\gamma > 0$ and $|x^G| \geq |G|^{\gamma}$, $x$ is not the identity and so $\prod_{i=1}^t r_i! \leq |C_G(x)| \leq |G|^{1-\gamma}$. We obtain $$(s/3t)^s \leq s!/t^s \leq \prod_{i=1}^t r_i! \leq |C_G(x)| \leq (n!/2)^{1-\gamma} < n^{n(1-\gamma)}.$$ Taking natural logarithms we obtain $s \log(s/3t) < n \left( 1-\gamma \right) \log(n)$. Since $s > 3t$, we obtain $$s < n(1-\gamma) \frac{\log(n)}{\log(s/3t)}.$$ Let $$\mathcal{A} := \{(n,s) \in \mathbb{N} \times \mathbb{N}\ :\ \exists x \in G\ :\ |x^G| \geq |G|^{\gamma}\ \mbox{and}\ s \geq n(1-\varepsilon_1)(1-\gamma+\varepsilon)\}.$$In order to prove the result, it is enough to show that $\mathcal{A}$ is finite. Assume by contradiction that $\mathcal{A}$ is infinite. Observe that, if $(n,s) \in \mathcal{A}$, then
\begin{align*}
\frac{\log(n)}{\log(s/3t)} & = \frac{\log(n)}{\log(s)-\log(3t)} \leq \frac{\log(n)}{\log(n)+\log((1-\varepsilon_1)(1-\gamma+\varepsilon))-\log(3t)}.
\end{align*}
It follows that $\log(n)/\log(s/3t)$ tends to $1$ as $(n,s) \in \mathcal{A}$ and $n$ goes to infinity, since $\varepsilon_1$, $\gamma$, $\varepsilon$ and $t$ are fixed. 

Since $\frac{\varepsilon(1-\varepsilon_1)}{1-\gamma} > \frac{\varepsilon}{1-\gamma+\varepsilon} > \varepsilon_1$, there exists $\varepsilon_2$ such that $0 < \varepsilon_2 < \frac{\varepsilon(1-\varepsilon_1)}{1-\gamma}-\varepsilon_1$. If $(n,s) \in \mathcal{A}$ and $n$ is sufficiently large, we have $\log(n)/\log(s/3t) \leq 1+\varepsilon_2$, therefore
\begin{align*}
s & < n(1-\gamma)(1+\varepsilon_2) < n(1-\gamma) \left( 1+\frac{\varepsilon(1-\varepsilon_1)}{1-\gamma}-\varepsilon_1 \right) = n (1-\varepsilon_1)(1-\gamma+\varepsilon).
\end{align*}
This contradicts the fact that $(n,s) \in \mathcal{A}$.
\end{proof}

\section{Background on character sums} \label{chsum}

Let $\Irr(G)$ denote the set of irreducible complex characters of a finite group $G$.

If $A,B$ are conjugacy classes of $G$ we are interested to know which conjugacy classes the normal set $AB$ contains. Fix $a \in A$, $b \in B$, $g \in G$. Then \cite[page 43]{ASH} gives $$|\{(x,y) \in A \times B\ :\ xy=g\}| = \frac{|A||B|}{|G|} \left( \sum_{\chi \in \Irr(G)} \frac{\chi(a)\chi(b)\overline{\chi(g)}}{\chi(1)} \right).$$It follows that the conjugacy class of $g$ in $G$ is contained in the normal set $AB$ if and only if 
\begin{equation} \label{cs}
\sum_{\chi \in \Irr(G)} \frac{\chi(a)\chi(b) \overline{\chi(g)}}{\chi(1)} \neq 0.
\end{equation}
Assume $G$ is the alternating group $\Alt(n)$. The strategy to show that condition (\ref{cs}) holds will often be the following: first, we separate the contribution of the trivial character in the sum, which is $1$, then we show that the remaining part tends to $0$ when $n$ tends to infinity. This implies that in this case condition (\ref{cs}) holds when $n$ is sufficiently large.

\section{Background on characters of $\Alt(n)$} \label{symalt}

In this section we review some basic facts about the characters of the alternating and symmetric groups. Everything here may be found in \cite[Chapter 2]{JamesKerber}.

A partition $\lambda = (\lambda_1,\ldots,\lambda_t)$ of $n$ is a sequence of positive integers $\lambda_1 \geq \ldots \geq \lambda_t$ such that $\lambda_1+\cdots+\lambda_t=n$. The partitions of $n$ correspond bijectively to the cycle structures of the elements of $\Sym(n)$ and to the Young diagrams of size $n$. Each partition $\lambda$ of $n$ determines uniquely a complex irreducible character $\chi_{\lambda}$ of $\Sym(n)$, and these are precisely the complex irreducible characters of $\Sym(n)$. We will use the well-known Murnaghan-Nakayama rule to compute character values, and the well-known hook length formula to compute character degrees.

We define $l$ to be $n$ if $n$ is odd and $n-1$ if $n$ is even. An $l$-hook will be a hook of length $l$. An immediate consequence of the Murnaghan-Nakayama rule is the following. Assume that $\lambda$ is a partition of $n$ and $g \in \Sym(n)$ is an $l$-cycle. If $\lambda$ contains an $l$-hook then $\chi_{\lambda}(g)=(-1)^k$ where $k$ is the leg length of the unique $l$-hook of $\lambda$. If $\lambda$ does not contain an $l$-hook, then $\chi_{\lambda}(g)=0$.

Before describing the irreducible complex characters of $\Alt(n)$ we introduce a notation. The conjugacy class of an element $g$ of $\Alt(n)$ may or may not be equal to its conjugacy class in $\Sym(n)$. If it is (resp. if it is not), we call $g$ a non-exceptional (resp. exceptional) element, the conjugacy class of $g$ a non-exceptional (resp. exceptional) class and the cycle type of $g$ a non-exceptional (resp. exceptional) cycle type. Since partitions of $n$ correspond bijectively to cycle types, we may also talk about exceptional and non-exceptional partitions. Recall that an element of $\Alt(n)$ is exceptional if and only if the lengths of the cycles in its disjoint cycle decomposition are odd and pairwise distinct (including $1$-cycles), and that the conjugacy class in $\Sym(n)$ of an exceptional element is the union of precisely two conjugacy classes of $\Alt(n)$ of equal size. An important example of an exceptional element is given by any $l$-cycle, where $l$ is defined in the previous paragraph.

The irreducible complex characters of $\Alt(n)$ are described as follows. Let $\lambda$ be a partition of $n$. Denote by $\lambda'$ the partition adjoint to $\lambda$, obtained by reflecting its Young diagram through the main diagonal. The partition $\lambda$ is said to be self-adjoint if $\lambda = \lambda'$. If $\lambda \neq \lambda'$, then the restriction of the character $\chi_{\lambda}$ to $\Alt(n)$ is an irreducible character of $\Alt(n)$, which we denote by $\psi_{\lambda}$. Clearly, $\psi_{\lambda} = \psi_{\lambda'}$ in this case. If $\lambda = \lambda'$, then the restriction of the character $\chi_{\lambda}$ to $\Alt(n)$ is the sum of two irreducible characters $\psi_{\lambda}^+$, $\psi_{\lambda}^-$ of $\Alt(n)$. Every irreducible complex character of $\Alt(n)$ is of the form $\psi_{\lambda}$ where $\lambda$ is a non-self-adjoint partition or $\psi_{\lambda}^{\pm}$ where $\lambda$ is a self-adjoint partition.

Given a self-adjoint partition $\lambda$ of $n$, denote by $h(\lambda)$ the partition whose parts are the lengths of those hooks of $\lambda$ whose heads are in the main diagonal. Observe that $h(\lambda)$ is an exceptional partition. For example, if $\lambda$ is a self-adjoint hook, then $h(\lambda)$ is the partition $(n)$.

Let $\lambda$ be a partition of $n$ and let $x \in \Alt(n)$. If $\lambda \neq \lambda'$, then $\psi_{\lambda}(x) = \chi_{\lambda}(x)$. If $\lambda = \lambda'$ and the cycle type of $x$ is not $h(\lambda)$, then $\psi_{\lambda}^{\pm}(x) = \chi_{\lambda}(x)/2$. If $\lambda = \lambda'$ and the cycle type of $x$ is $h(\lambda)$, then $x$ is an exceptional element, hence there exists $y \in \Sym(n)$ which is conjugate to $x$ in $\Sym(n)$ but not in $\Alt(n)$. Let $h^{\lambda}_{ii}$ be the length of the hook in $\lambda$ with head in position $(i,i)$. We have $$\psi_{\lambda}^{\pm}(x) = \frac{1}{2} \left( \chi_{\lambda}(x) \pm \sqrt{\chi_{\lambda}(x) \prod_{i}h^{\lambda}_{ii}} \right), \hspace{.3cm} \psi_{\lambda}^{\pm}(y) = \frac{1}{2} \left( \chi_{\lambda}(y) \mp \sqrt{\chi_{\lambda}(y) \prod_{i}h^{\lambda}_{ii}} \right).$$

An important consequence for us is the following. Let $l$ be $n$ if $n$ is odd and $n-1$ if $n$ is even and let $x$ be an $l$-cycle. Let $\lambda$ be a partition of $n$. If $\lambda$ does not contain an $l$-hook, then $\psi_{\lambda}(x) = 0$. Assume $\lambda$ contains an $l$-hook. Observe that in this case $\lambda$ contains a unique $l$-hook. If $\lambda \neq \lambda'$, then $\psi_{\lambda}(x) = \chi_{\lambda}(x) = (-1)^k$ where $k$ is the leg length of the $l$-hook contained in $\lambda$. If $\lambda = \lambda'$, then $|\psi_{\lambda}^{\pm}(x)| \leq \sqrt{n}$.

\section{Proof of Theorem \ref{main}} \label{proofmain}

In this section we prove Theorem \ref{main}. First we need a list of technical lemmas.

The following is a special case of \cite[Theorem 1.2]{LarsenShalev}.

\begin{theorem}[Larsen, Shalev] \label{ls}
If $\sigma \in \mathrm{Sym}(n)$ has at most $n^{o(1)}$ cycles of length at most $5$, then $|\chi(\sigma)| \leq \chi(1)^{1/5 + o(1)}$ for every complex irreducible character $\chi$ of $\mathrm{Sym}(n)$. 
\end{theorem}

This theorem will be applied in the special case when $\sigma$ is an exceptional element in $\Sym(n)$, and $\chi=\chi_{\lambda}$, where $\lambda$ is a partition containing an $n$-hook for $n$ odd and an $(n-1)$-hook for $n$ even.

Let $l$ be $n$ if $n$ is odd and $n-1$ if $n$ is even.

\begin{lemma} \label{degalt}
Let $n \geq 9$ and let $G=\Alt(n)$. Let $\psi \in \Irr(G)$ be a nontrivial character associated to a partition $\lambda$ of $n$ containing an $l$-hook. Either $\psi(1) \geq n(n-3)/2$ or $n$ is odd, $\psi$ is equal to the restriction of $\chi_{\lambda}$ to $G$ where $\lambda=(n-1,1)$ and $\psi(1)=n-1$. Moreover if $\lambda=\lambda'$, then $\psi(1) \geq 2^{n-2}/n^2$.
\end{lemma}

\begin{proof}
Since $\psi$ is nontrivial, both $\lambda$ and $\lambda'$ are different from $(n)$. Let $\lambda=\lambda'$. The hook length formula implies $$\psi(1) = \left\{ \begin{array}{ll} \frac{1}{2} \binom{n}{n/2} \frac{(n/2-1)^2}{n-1} & \mbox{if}\ n\ \mbox{is even}, \\ & \\ \frac{1}{2} \binom{n-1}{(n-1)/2} & \mbox{if}\ n\ \mbox{is odd}. \end{array} \right.$$It follows that $\psi(1) \geq n(n-3)/2$ for $n \geq 9$. In any case $\psi(1) \geq 2^{n-2}/n^2$. Now let $\lambda \neq \lambda'$. Let $k \geq 1$ be the leg length of the unique $l$-hook contained in $\lambda$. We will use the hook length formula to compute $\psi(1)$. If $n$ is odd, then $\psi(1) = \binom{n-1}{k}$. This is either $n-1$ or at least $\binom{n-1}{2}$. Let $n$ be even. Without loss of generality, $k$ satisfies $1 \leq k \leq n/2-1$ and $$\psi(1) = d(n,k) := \binom{n}{k+1} \frac{k(n-k-2)}{n-1}.$$ We claim that $d(n,k) \geq \binom{n-1}{2}$. Observe that if $k \geq 3$ we have $k(n-k-2) \geq n-3$ whenever $n \geq 6$, hence $d(n,k) \geq d(n,1)$, and if $n \geq 8$ then $$d(n,1) \leq d(n,2) \leq d(n,3).$$ This implies that if $n \geq 8$ then, for any $k$ between $1$ and $n/2-1$, we have $d(n,k) \geq d(n,1) = n(n-3)/2$.
\end{proof}

Denote the set of all $l$-cycles in $\Sym(n)$ by $O_l$.

\begin{lemma} \label{excon}
Let $A$ and $B$ be conjugacy classes of $G=\Alt(n)$, and let $n$ be sufficiently large.
\begin{enumerate}
\item If $A$ and $B$ are exceptional classes, then $AB$ contains $O_l$.
\item If $A$ and $B$ are classes of $l$-cycles, then $AB$ contains every exceptional class.
\item If $A$ is a class of $l$-cycles, then $O_lA=G$.
\item The product of any three classes of $l$-cycles equals $G$.
\end{enumerate}
\end{lemma} 

\begin{proof}
We prove part (1). Assume $A$ and $B$ are exceptional classes of $G$ and let $D$ be the conjugacy class in $G$ of an $l$-cycle. Let $x \in A$, $y \in B$, and $g \in D$. Let $\psi \in \Irr(G)$ be arbitrary, and let $\lambda$ be the partition of $n$ associated to $\psi$. Then $\psi(g) = 0$, unless the Young diagram of $\lambda$ contains an $l$-hook, by Section \ref{symalt}. Therefore, it is enough to show that (cf. Section \ref{chsum}) the rational number $$\Sigma := \sum_{\psi} \frac{\psi(x)\psi(y) \overline{\psi(g)}}{\psi(1)}$$ tends to $0$ as $n$ tends to infinity, where the sum is over the nontrivial irreducible characters of $G$ corresponding to partitions containing an $l$-hook. Let $z$ be any of $x$, $y$, or $g$. For such characters $\psi$ we have $|\psi(z)| \leq |\chi_{\lambda}(z)|$, unless $\lambda$ is self-adjoint and $z$ is an $l$-cycle. Let $\lambda$ be self-adjoint, and assume $z$ is an $l$-cycle. We have $|\psi(z)| \leq \sqrt{n}$, which is at most ${\psi(1)}^{o(1)}$ as $n \to \infty$ by Lemma \ref{degalt}. Since $x,y,g$ are exceptional elements, we obtain the following by Theorem \ref{ls} and Lemma \ref{degalt} for $n$ sufficiently large. 
\begin{align*}
|\Sigma| & \leq \sum_{\psi} \Big| \frac{\psi(x) \psi(y) \overline{\psi(g)}}{\psi(1)} \Big| = \sum_{\psi} \frac{|\psi(x)| |\psi(y)| |\psi(g)|}{\psi(1)} \leq \sum_{\psi} \psi(1)^{-3/5+o(1)} \\ & \leq (n-1)^{-3/5+o(1)}+n \left( \frac{n(n-3)}{2} \right)^{-3/5+o(1)},
\end{align*}
where the sums are over the nontrivial irreducible characters $\psi$ of $G$ associated to partitions containing an $l$-hook. Therefore $|\Sigma|$ tends to $0$ as $n \to \infty$.

We prove part (2). Assume $A$, $B$ are classes of $l$-cycles and $C$ is any exceptional class. We need to show that $C \subseteq AB$. Since the conjugacy class $B^{-1}$ of the inverse of an element of $B$ is an exceptional class, $A \subseteq CB^{-1}$ by part (1). Fix $a \in A$. There exist $b \in B$, $c \in C$ such that $a = cb^{-1}$, so that $c=ab \in AB$. Since $C$ is the conjugacy class of $c$ and $AB$ is a normal set, $C \subseteq AB$ follows.

We prove part (3). Let $A$ be a class of $l$-cycles. Then, since $O_lO_l=G$ (by Lemma \ref{on2an}), $O_lA$ contains a representative of every conjugacy class of $\Sym(n)$, therefore $O_lA$ contains every non-exceptional class. On the other hand, $O_l$ contains a class of $G$ consisting of $l$-cycles, hence $O_lA$ also contains every exceptional class by part (2). It follows that $O_lA=G$.

We prove part (4). Let $A,B,C$ be classes of $l$-cycles. Then $AB$ contains $O_l$ by part (1), so $ABC=G$ by part (3).
\end{proof}

\begin{lemma} \label{onsa}
Let $A,B$ be conjugacy classes of $G=\Alt(n)$. If neither $A$ nor $B$ is contained in $O_l$ and $AB$ contains an $l$-cycle, then $AB$ contains $O_l$.
\end{lemma}

\begin{proof}
Let $x \in A$, $y \in B$. Choose two nonconjugate $l$-cycles $d_1$, $d_2$ of $G$ such that $d_1 \in AB$. We need to show that $d_2$ belongs to $AB$. Let $$a_i := \sum_{\psi \in \Irr(G)} \frac{\psi(x) \psi(y) \overline{\psi(d_i)}}{\psi(1)}.$$As explained in Section \ref{chsum}, the fact that $d_1 \in AB$ is equivalent to saying that $a_1 \neq 0$. We need to show that $a_2 \neq 0$. We will show that $a_1=a_2$.

Section \ref{symalt} implies that in the formula that defines $a_i$ the summation can be done only over those $\psi$ which are labelled by partitions containing an $l$-hook. If $\lambda$ is not a self-adjoint partition containing an $l$-hook, then there is precisely one irreducible character $\psi$ of $G$ associated to $\lambda$ and the corresponding summand is the same in $a_1$ and in $a_2$. Assume now that $\lambda=\lambda_0$ is the unique self-adjoint partition of $n$ containing an $l$-hook, so that the restriction of $\chi_{\lambda_0}$ to $G$ is a sum of two irreducible characters $\psi_{\lambda_0}^+$ and $\psi_{\lambda_0}^-$. Since $x,y,1$ are not of cycle type $h(\lambda_0)$, we have $\psi_{\lambda_0}^+(x) = \psi_{\lambda_0}^-(x)$, $\psi_{\lambda_0}^+(y) = \psi_{\lambda_0}^-(y)$ and $\psi_{\lambda_0}^+(1) = \psi_{\lambda_0}^-(1)$. It follows that $$a_1-a_2 = \frac{\psi_{\lambda_0}^+(x) \psi_{\lambda_0}^+(y)}{\psi_{\lambda_0}^+(1)} \cdot \left( \overline{\psi_{\lambda_0}^+(d_1)}+\overline{\psi_{\lambda_0}^-(d_1)}-\overline{\psi_{\lambda_0}^+(d_2)}-\overline{\psi_{\lambda_0}^-(d_2)} \right).$$The second factor equals $\overline{\chi_{\lambda_0}(d_1)-\chi_{\lambda_0}(d_2)}$, which is equal to $0$ since $d_1$ and $d_2$ are conjugate in $\Sym(n)$.
\end{proof} 

\begin{lemma} \label{abon}
Let $\varepsilon > 0$. There exists $N \in \mathbb{N}$ such that the following holds for every $n \geq N$: whenever $A$ and $B$ are conjugacy classes of $G = \Alt(n)$ not contained in $O_l$ and $|A||B| \geq |G|^{1+\varepsilon}$, then $AB \supseteq O_l$. 
\end{lemma}

\begin{proof}
If both $A$ and $B$ are exceptional, then the result follows from Lemma \ref{excon}(1). Assume without loss of generality that $A$ is non-exceptional. Let $B'$ be the conjugacy class of $\Sym(n)$ containing $B$. We claim that $AB' \supseteq O_l$. Write $|A|=|G|^{\gamma_1}$, $|B|=|G|^{\gamma_2}$, so that $0 < \gamma_i < 1$ for $i=1,2$ and $\gamma_1+\gamma_2 \geq 1+\varepsilon$. Lemma \ref{deltac} implies that if $n$ is large enough, then $\delta(A) > n (\gamma_1-\varepsilon/2)$ and $\delta(B) > n (\gamma_2-\varepsilon/2)$, therefore $$\delta(A)+\delta(B') = \delta(A)+\delta(B) > n(\gamma_1+\gamma_2-\varepsilon/2-\varepsilon/2) \geq n(1+\varepsilon-\varepsilon) = n,$$which implies that $AB'$ contains $O_l$ by Theorem \ref{dr}. We may assume that $B'$ properly contains $B$, that is, $B$ is an exceptional class. Since $AB'$ contains $O_l$, the normal set $AB$ contains an $l$-cycle, hence $AB$ contains $O_l$ by Lemma \ref{onsa}.
\end{proof}

\begin{lemma} \label{thcl}
Theorem \ref{main} holds in the case when $A$, $B$, $C$, $D$ are conjugacy classes.
\end{lemma}

\begin{proof}
Let $l$ be $n$ if $n$ is odd and $n-1$ if $n$ is even. If at least three of the classes $A$, $B$, $C$, $D$ consist of $l$-cycles, then the result follows from Lemma \ref{excon}(4). Therefore we may assume that $A$ and $B$ are not classes of $l$-cycles, so that $O_l \subseteq AB$ by Lemma \ref{abon}. If any of $C,D$ is a class of $l$-cycles, then the result follows from Lemma \ref{excon}(3) and if $C$ and $D$ are not classes of $l$-cycles, then the result follows from Lemma \ref{abon} and Lemma \ref{on2an}.
\end{proof}

\begin{proof}[Proof of Theorem \ref{main}]
Let $|A|=|G|^a$, $|B|=|G|^b$, $|C|=|G|^c$, $|D|=|G|^d$. For any $\gamma > 0$ there exists $N=N(\gamma) \in \mathbb{N}$ such that whenever $n > N$, any normal subset $S$ of $G$ of size at least $|G|^{\gamma}$ contains a conjugacy class of $G$ of size at least $|G|^{\gamma-\varepsilon/3}$ by \cite[Lemma 4.2]{MP}. Applying this for $n > \max \{N(a),N(b),N(c),N(d)\}$ and for $(S,\gamma) = (A,a)$, $(B,b)$, $(C,c)$, $(D,d)$ we obtain that there exist conjugacy classes $A_0$, $B_0$, $C_0$, $D_0$ of $G$ contained in $A$, $B$, $C$, $D$ respectively, such that $|A_0| \geq |G|^{a-\varepsilon/3}$, $|B_0| \geq |G|^{b-\varepsilon/3}$, $|C_0| \geq |G|^{c-\varepsilon/3}$ and $|D_0| \geq |G|^{d-\varepsilon/3}$. The hypotheses of Theorem \ref{main} hold for $A_0$, $B_0$, $C_0$, $D_0$ with the constant $\varepsilon/3$. Hence $ABCD \supseteq A_0B_0C_0D_0=G$ by Lemma \ref{thcl}.
\end{proof}

\end{document}